\documentclass[10pt]{extarticle}
\usepackage{theoremref}
\usepackage{etoolbox}
\usepackage[utf8]{inputenc}
\usepackage[T1]{fontenc}
\usepackage[british]{babel} 
\usepackage{fullpage}
\usepackage{enumitem}
\usepackage{amsmath,amssymb,bbm}
\usepackage{amsthm}
\newtheorem{definition}{Definition}[section]
\newtheorem{thm}[definition]{Theorem}
\newtheorem{lemma}[definition]{Lemma}

\newtheorem{fact}[definition]{Fact}

\AtBeginDocument{
	\title{The structure of typical eye-free graphs
and a Tur\'an-type result for two weighted colours}
	\date{}
  \author{Peter Keevash\thanks{ 
Mathematical Institute, University of Oxford, Oxford, UK. 
Email: keevash@maths.ox.ac.uk.
Research supported in part by 
ERC Consolidator Grant 647678.} 
 \and William Lochet\thanks{
Computer Science Department, Ecole Normale Sup\'erieure de Lyon, 
France. Email: william.lochet@ens-lyon.fr.}
}  
}

\newcommand{\mc}[1]{\mathcal{#1}}
\newcommand{\mb}[1]{\mathbb{#1}}
\newcommand{\nib}[1]{\noindent {\bf #1}}
\newcommand{\nim}[1]{\noindent {\em #1}}
\newcommand{\brac}[1]{\left( #1 \right)}

\newcommand{\bcl}[1]{\left\lceil #1 \right\rceil}

\newcommand{\sub}{\subseteq}

\newcommand{\sm}{\setminus}
\newcommand{\ov}{\overline}

\newcommand{\eps}{\varepsilon}
\newcommand{\es}{\emptyset}

\newcommand{\aA}{\alpha}
\newcommand{\bB}{\beta}

\newcommand{\dD}{\delta}
\newcommand{\kK}{\kappa}
\newcommand{\lL}{\lambda}

\newcommand{\sS}{\sigma}

\newcommand{\DD}{\Delta}

\begin{document}

\maketitle

\begin{abstract}
The $(a,b)$-eye is the graph $I_{a,b} = K_{a+b} \sm K_b$
obtained by deleting the edges of a clique of size $b$ from a clique of size $a+b$.
We show that for any $a,b \ge 2$ and $p \in (0,1)$,
if we condition the random graph $G \sim G(n,p)$ on having no induced copy of $I_{a,b}$,
then with high probability $G$ is close to an $a$-partite graph
or the complement of a $(b-1)$-partite graph.
Our proof uses the recently developed theory of hypergraph containers,
and a stability result for an extremal problem with two weighted colours.
We also apply the stability method to obtain an exact 
Tur\'an-type result for this extremal problem.
\end{abstract}

\section{Introduction}

The typical structure of graphs in a monotone or hereditary property
has been the subject of much recent research in Extremal Combinatorics.
Let us first consider a monotone graph property $\mc{P}$, or equivalently, 
the set of graphs $G$ that are \emph{$\mc{H}$-free} for some family of graphs $\mc{H}$,
meaning that $G$ does not contain any $H \in \mc{H}$ as a subgraph.
Let $\mc{P}_n$ denote the set of graphs in $\mc{P}$ with vertex set $[n] = \{1,\dots,n\}$.
A prerequisite for understanding a typical graph in $\mc{P}_n$ is an estimate for $|\mc{P}_n|$,
which in turn is closely linked to $\max_{G \in \mc{P}_n} e(G)$. This leads us back to the
classical theorem of Tur\'an \cite{Turan}, who showed that for $t \ge 2$, the maximum number of edges 
in a $K_{t+1}$-free graph on $[n]$ is uniquely achieved by the complete $t$-partite graph
that is balanced, meaning that its part sizes are as equal as possible.
This implies $|\mc{P}_n| \ge 2^{(1-1/t+o(1))n^2/2}$ when $\mc{P}$ is the property of being $K_{t+1}$-free;
an upper bound of the same form was proved by Erd\H{o}s, Kleitman and Rothschild \cite{EKR}.
More generally, for any $\mc{H}$, Erd\H{o}s, Frankl and R\"odl \cite{EFR}
showed $|\mc{P}_n| = 2^{(1-1/(\chi(\mc{H})-1)+o(1))n^2/2}$,
where the \textit{chromatic number} $\chi(H)$ of a graph $H$ is the minimum $t$ 
such that $H$ is $t$-partite, and $\chi(\mc{H}) = \min_{H \in \mc{H}} \chi(H)$;
here the extremal result is due to Erd\H{o}s, Stone and Simonovits \cite{ES46,ES66}.
Moreover, Kolaitis, Pr\"omel and Rothschild \cite{KPR} determined the typical structure 
of $K_{t+1}$-free graphs: they showed that almost all graphs in $\mc{P}_n$ are $t$-partite.

More recently, Balogh, Morris, Samotij and Warnke \cite{BMSW} considered the same problem
for $K_{t+1}$-free graphs $\mc{P}_{n,m}$ in which we also fix the number $m$ of edges: 
they determined the `critical interval' $I$ for which (roughly speaking)
when $m \in I$ most graphs in $\mc{P}_{n,m}$ are $t$-partite,
and when $m \notin I$ most graphs in $\mc{P}_{n,m}$ are not $t$-partite.
For more general $\mc{H}$ much less is known, particularly if $\mc{H}$
contains a bipartite graph, as we do not even know $\max_{G \in \mc{P}_n} e(G)$.
Results that are as precise as can be expected given this uncertainty
were given by Balogh, Bollob\'as and Simonovits \cite{BBS}:
informally speaking, they show that for a typical $\mc{H}$-free graph
one can delete a constant number of vertices and decompose the rest
into a constant number of parts which are $\mc{M}$-free,
for some (specific) family $\mc{M}$ that contains a bipartite graph.

Now consider a hereditary graph property $\mc{P}$, or equivalently, 
the set of graphs $G$ that are \emph{$\mc{H}$-ifree} for some family of graphs $\mc{H}$,
meaning that $G$ does not contain any $H \in \mc{H}$ as an induced subgraph.
Here the result of Erd\H{o}s, Frankl and R\"odl was generalised
by Alekseev \cite{A} and Bollob\'as and Thomason \cite{BT}:
they showed that $|\mc{P}_n| = 2^{(1-1/\chi_c(\mc{P})+o(1))n^2/2}$,
where $\chi_c(\mc{P})$ is the `colouring number' of $\mc{P}$
(we will not give the definition here, but it is implicit 
in the theory of `types' discussed below).
The general theory was substantially developed by Pr\"omel and Steger (see e.g. \cite{PS1,PS2});
among other things they determined the typical structure of $C_4$-ifree graphs and $C_5$-ifree graphs.
An analogue of the result of Balogh, Bollob\'as and Simonovits (weaker in some details) 
was given by Alon, Balogh, Bollob\'as and Morris \cite{ABBM}.

However, we do not know yet results for hereditary properties
analogous to that of Balogh, Morris, Samotij and Warnke.
Instead of random $\mc{H}$-ifree graphs on $[n]$ with $m$ edges,
one can consider the somewhat related model, which we denote by $G_{n,p}[\mc{H}]$,
obtained by conditioning the Erd\H{o}s-R\'enyi random graph $G_{n,p}$ on being $\mc{H}$-ifree
(in the concluding remarks we will compare the two models
and explain why the latter is easier to analyse).
Note that in the case $p=1/2$ this is a uniformly random $\mc{H}$-ifree graph,
so the typical structure is described by the result of Alon, Balogh, Bollob\'as and Morris.
For general $p$, the first step towards analysing the model
is an estimate for the probability that $G_{n,p}$ is $\mc{H}$-ifree;
this is provided by results of Bollob\'as and Thomason \cite{BT} (see Subsection \ref{igraphs}),
refined and corrected by Marchant and Thomason \cite{Mar11,MT} (see Subsections \ref{types} and \ref{extensions}).

\subsection{The typical structure of eye-free graphs} \label{main1}

Now we will state our first main result, concerning the typical structure of graphs in $G_{n,p}[I_{a,b}]$,
i.e.\ $G_{n,p}$ conditioned on being $I_{a,b}$-ifree, where $I_{a,b} = K_{a+b} \sm K_b$.
First we note two natural constructions of $I_{a,b}$-ifree graphs:
any graph $G$ with $\chi(G) \le a$ or $\chi(\ov{G}) < b$ is $I_{a,b}$-ifree,
where $\ov{G}$ denotes the \emph{complement} graph on $V(G)$, in which a pair of vertices
is an edge of $\ov{G}$ if and only if it is not an edge of $G$.

\begin{thm} \thlabel{eye-free} 
For all $a,b \ge 2$, $p \in (0,1)$ and $\eps>0$ there is $n_0$ such that if $n>n_0$ 
and $G \sim G_{n,p}[I_{a,b}]$ then with probability at least $1-\eps$,
there is $G'$ with $|G \triangle G'| < \eps n^2$ such that $\chi(G')=a$ or $\chi(\ov{G'})=b-1$.
\end{thm}

Note that there is no loss of generality in Theorem \ref{eye-free} in assuming $a \ge 1$ and $b \ge 2$;
otherwise $I_{a,b}$ is a clique or an independent set, which are cases covered by results mentioned above.
One may still wonder about the assumption $a \ge 2$. We do not know whether the result holds for $a=1$,
and in fact, we will see below that there are good reasons to think that it does not.
Henceforth we assume $a,b \ge 2$.

\subsection{Hereditary properties and extremal igraph theory} \label{igraphs}

One can reformulate containment of induced subgraphs using two-coloured complete graphs.
Given a graph $H$, let $c(H)$ be the two-coloured complete graph on $V(H)$,
where edges of $H$ are coloured red and non-edges of $H$ are coloured blue.
Given a complete graph $C$ with edges coloured red or blue, we write $C_r$ for the graph 
of red edges and $C_b$ for the graph of blue edges. Then $H$ is an induced subgraph of $G$
if and only if we can identify $V(H)$ with a subset of $V(G)$
so that $c(H)_r \sub c(G)_r$ and $c(H)_b \sub c(G)_b$.

To represent natural families of graphs with forbidden induced subgraphs we introduce 
a third colour, green, which should be thought of as a `wildcard', meaning `red or blue'. 
Suppose $C$ is a complete graph with edges coloured red, blue or green;
for brevity we will call $C$ an \textit{igraph}
(the `i' is for `induced', looking ahead to Theorem \ref{containers}).
We also view $C$ as a multigraph with edges coloured red or blue,
in which there is one edge for any edge coloured red or blue,
but two edges (one red and one blue) for any edge coloured green.
We stress that the meaning of `red' and `blue' depends on which viewpoint we adopt:
in the three-coloured viewpoint `red' means `red' and `blue' means `blue';
in the two-coloured viewpoint `red' means `red or green' and `blue' means `blue or green'.

Suppose $C$ is an igraph. We let $C^-_r$, $C^-_b$, $C_g$ respectively
be the graphs of red, blue, green edges in the three-coloured viewpoint.
We let $C_r = C^-_r \cup C_g$ and $C_b = C^-_b \cup C_g$ 
be the graphs of red and blue edges in the two-coloured viewpoint.
We let $\mc{G}(C)$ be the set of graphs $G$ on $V(C)$
with $C^-_r \sub G$ and $C^-_b \sub \ov{G}$.

We say that $C$ contains a graph $H$ and write $H \sub C$
if we can identify $V(H)$ with a subset of $V(C)$
so that $H \sub C_r$ and $\ov{H} \sub C_b$.
Equivalently, $C$ contains $H$ if it contains $c(H)$ in the two-coloured viewpoint,
so henceforth we simplify notation by using $H$ to denote $H$ or $c(H)$ depending on the context.
We say that $C$ is $H$-free if it does not contain $H$.
Note that $C$ is $H$-free if and only if
$H$ is not an induced subgraph of any $G \in \mc{G}(C)$.
For a family of graphs $\mc{H}$ we say that $C$ is $\mc{H}$-free
if it does not contain any $H \in \mc{H}$.

Now, following Richer \cite{R} and Marchant and Thomason \cite{MT},
we consider the following extremal problem for igraphs, 
which is useful for counting $H$-free graphs,
and is also of interest in its own right
as a generalisation of extremal graph theory.
Given an igraph $C$ and $p \in [0,1]$,
we define the \textit{$p$-weight} of $C$ by 
\[w_p(C) = p|C_r| + (1-p)|C_b| = p|C^-_r| + (1-p)|C^-_b| + |C_g|.\]
Equivalently, the $p$-weight of $C$ is the total weight of all edges, where,
in the two-coloured viewpoint, red edges have weight $p$ and blue edges have weight $1-p$.
Given a graph family $\mc{H}$ and $n \ge 1$, we let $\text{kex}_p(\mc{H},n)$
be the maximum of $w_p(C)$ over all $\mc{H}$-free igraphs $C$ on $[n]$.
The \textit{induced $p$-density} of $\mc{H}$ 
is $\kK_p(\mc{H}) = \lim_{n \to \infty} \tbinom{n}{2}^{-1} \text{kex}_p(\mc{H})$
(it is not hard to see that the limit exists).
We may omit $p$ from our notation if it is clear from the context.
When $\mc{H} = \{H\}$ is a single graph we identify it with $H$.

To see the connection of hereditary properties to extremal igraph problems,
we note that if $G \sim G_{n,p}$ and $C$ is an igraph on $[n]$
then $\mb{P}(G \sub C) = p^{|C^-_r|} (1-p)^{|C^-_b|} = 2^{-H_p(C)}$, defining
\[ H_p(C) = - |C^-_r|\log_2 p - |C^-_b|\log_2 (1-p)
= - (\log_2 p + \log_2 (1-p)) (\tbinom{n}{2} - w_{p'}(C)),\]
where $p' = \tfrac{\log_2 (1-p)}{\log_2 p + \log_2 (1-p)}$. Thus 
\[\mb{P}(G \text{ is } \mc{H}\text{-ifree}) 
\ge 2^{(\log_2 p + \log_2 (1-p)) (1-\kK_{p'}(\mc{H})+o(1)) \tbinom{n}{2}}.\]
On the other hand, an upper bound of the same form
follows from a result of Bollob\'as and Thomason \cite{BT}.

When viewed as an extremal problem for two-coloured multigraphs,
it is perhaps more natural to allow incomplete extremal graphs,
i.e.\ to allow pairs that are neither red nor blue.
This version of the problem was also considered by Richer \cite{R} (see also \cite{MT})
and is the subject of a conjecture of Diwan and Mubayi \cite{DM} (see the concluding remarks).
To model it by colours, we define a \emph{whitened igraph} $C$
to be a complete graph with edges coloured red, blue, green or white;
the interpretation of a white pair is that it is `missing':
it does not contribute to $w_p(C)$ and cannot be used to form a copy of $H$.
Even if we are only interested in igraphs, our use of the removal lemma
(see Subsection \ref{removal}) forces us to consider whitened igraphs.

\subsection{Extremal eye-free igraphs}

Henceforth we specialise to case when $H$ is the $(a,b)$-eye $I_{a,b}$, which we may think of as a graph,
or as a two-coloured $K_{a+b}$ where some $K_b$ is blue and the remaining edges are red. 
Marchant and Thomason \cite[Example 5.5]{MT} showed that 
\begin{equation}\label{eqk}
\kK_p(I_{a,b}) = \max\{1-\tfrac{p}{a},1-\tfrac{1-p}{b-1}\}.
\end{equation}
The lower bounds are given by the following two constructions
corresponding to those appearing in Theorem \ref{eye-free}.
We let $\mc{B}_a(n)$ be the set of igraphs $C$ on $[n]$
consisting of $a$ disjoint blue cliques joined by green edges,
and $B_a(n) \in \mc{B}_a(n)$ be the igraph 
in which the clique sizes are as equal as possible;
we define $\mc{R}_{b-1}(n)$ and $R_{b-1}(n)$ similarly,
replacing `blue' by `red'.

Note that $\chi(G) \le a$ if and only $G \in \mc{G}(C)$ for some $C \in \mc{B}_a(n)$
and $\chi(\ov{G}) \le b-1$ if and only $G \in \mc{G}(C)$ for some $C \in \mc{R}_{b-1}(n)$.
As described above, (\ref{eqk}) implies an asymptotic formula for $\log_2 \mb{P}(G_{n,p} \text{ is } I_{a,b}\text{-ifree})$.
We will obtain the more precise result in Theorem \ref{eye-free} from the following `supersaturated stability' theorem.
For the statement we introduce some notation describing the approximately extremal igraphs. Let
\begin{equation*}
\mc{G}_p(n) = \begin{cases} 
\mc{B}_a(n) & \text{ if } p<\tfrac{a}{a+b-1}, \\
\mc{R}_{b-1}(n) & \text{ if } p>\tfrac{a}{a+b-1}, \\
\mc{B}_a(n) \cup \mc{R}_{b-1}(n) & \text{ if } p=\tfrac{a}{a+b-1}. \\
\end{cases}
\end{equation*}

\begin{thm}\thlabel{superstability}
For every $\eps>0$, $a,b \ge 2$, and $p \in (0,1)$ there exists $n_0$ and $\dD>0$, 
such that for any igraph $G$ on $n > n_0$ vertices with fewer than 
$\dD n^{a+b}$ copies of $I_{a,b}$ and $w_p(G) > (\kK_p(I_{a,b}) - \dD)\tbinom{n}{2}$
we can modify at most $\eps n^2$ edges of $G$ to obtain $G' \in \mc{G}_p(n)$.
\end{thm}

We also deduce our second main result,
which is the following exact Tur\'an-type result for igraphs.

\begin{thm}\thlabel{exact}
For any $a,b \ge 2$, $p \in (0,1)$ and $n$ sufficiently large, 
the $I_{a,b}$-free igraphs on $[n]$ with maximum $p$-weight are
\begin{itemize}[nolistsep]
	\item $R_{b-1}(n)$ if $p > \tfrac{a}{a+b-1}$ or $p = \tfrac{a}{a+b-1}$ and $a < b-1$,
	\item $B_a(n)$ if $p < \tfrac{a}{a+b-1}$ or $p = \tfrac{a}{a+b-1}$ and $a > b-1$,
	\item $R_{b-1}(n)$ and $B_a(n)$ if $p = \tfrac{a}{a+b-1} = 1/2$.
\end{itemize}
\end{thm}

This paper is organised as follows. The next section contains various preliminary steps
that reduce Theorem \ref{eye-free} to Theorem \ref{superstability} and thence to
a stability result (Theorem \ref{stability}). We prove this stability result 
in Section 3 then use it to also deduce Theorem \ref{exact} in Section 4.
The final section contains some concluding remarks and open problems.

\section{Preliminaries}

In this section we gather various tools and preliminary steps for our arguments.
The first subsection reduces Theorem \ref{eye-free} to Theorem \ref{superstability}
via an igraph container theorem. The second reduces Theorem \ref{superstability} to 
a stability result (Theorem \ref{stability}) via an igraph removal lemma.
The third and fourth subsections present some of the theory of Marchant and Thomason \cite{MT};
in some instances we need modifications of their results, so we give proofs of these.

\subsection{Igraph containers} \label{containers}

We start by stating a consequence of the recent hypergraph container theorems obtained 
independently by Balogh, Morris and Samotij \cite{BMS} and by Saxton and Thomason \cite{ST}.

\begin{thm} \thlabel{containers} (\cite[Theorem 2.6]{ST}, case $\ell=2$) 
For any graph $H$ and $\dD>0$ there is $c>0$ such that for $n$ sufficiently large
there is a collection $\mc{C}$ of igraphs on $[n]$ such that
\begin{enumerate}[label=(\alph*)]
\item for every $H$-ifree graph $I$ on $[n]$ there is $C \in \mc{C}$ with $I \sub C$,
\item any $C \in \mc{C}$ contains at most $\dD n^{|V(H)|}$ copies of $H$,
\item $\log |\mc{C}| \le cn^{2-2/(|V(H)|+1)} \log n$.
\end{enumerate}
\end{thm}

As shown in \cite[Theorem 1.6]{ST}, it is not hard to give an alternative
proof of the result of Bollob\'as and Thomason using Theorem \ref{containers}.
Indeed, we can bound the probability that $G \sim G_{n,p}$ is $H$-ifree by $\sum_{C \in \mc{C}} \mb{P}(G \sub C)$.
From (b) one can deduce $w_{p'}(C) < (\kK_{p'}(H)+o(1)) \tbinom{n}{2}$ for every $C \in \mc{C}$, 
so by (c) we have $- \log_2 \mb{P}(G \text{ is } H\text{-ifree})
\ge \log_2 |\mc{C}| + \min_{C \in \mc{C}} H_p(C)
\ge (\log_2 p + \log_2 (1-p)) (1-\kK_{p'}(H)+o(1)) \tbinom{n}{2}$.
Similarly, assuming Theorem \ref{superstability}, we deduce Theorem \ref{eye-free}
from Theorem \ref{containers} as follows.

\nim{Proof of Theorem \ref{eye-free}.}
Let $\dD$ be provided by Theorem \ref{superstability} applied with 
$p' = \tfrac{\log_2 (1-p)}{\log_2 p + \log_2 (1-p)}$ in place of $p$, and let 
$\mc{C}$ be the collection provided by Theorem \ref{containers} applied with $H=I_{a,b}$.
Let $\mc{C}_0$ be the set of $C \in \mc{C}$
with $\tbinom{n}{2}^{-1} w_{p'}(C) \le \kK_{p'}(I_{a,b}) - \dD$. 
For $G \sim G_{n,p}$ and $n$ large we have 
\begin{align*}
\log_2 \sum_{C \in \mc{C}_0} \mb{P}(G \sub C)
& \le \log_2 |\mc{C}| +  \max_{C \in \mc{C}_0} (-H_p(C))  \\
& \le (\log_2 p + \log_2 (1-p)) (1-\kK_{p'}(I_{a,b})+\dD/2) \tbinom{n}{2}  \\
& < \log_2 \brac{ \eps\ \mb{P}(G \text{ is } I_{a,b}\text{-ifree}) }.
\end{align*}
Thus for $G \sim G_{n,p}[I_{a,b}]$, with probability at least $1-\eps$
we have $G \sub C$ for some $C \in \mc{C} \sm \mc{C}_0$.
By Theorem \ref{superstability}, for any such $C$
there exists $F \sub E(C)$ with $|F| < \eps n^2$ such that $C \triangle F$ 
can be partitioned into $a$ blue cliques joined by green edges 
or into $b-1$ red cliques joined by green edges.
Then by modifying $G$ on the pairs of $F$ we obtain $G'$
with $|G \triangle G'| < \eps n^2$ such that $\chi(G') \le a$ or $\chi(\ov{G'}) \le b-1$.
By modifying a constant number of edges we can assume equality holds. \qed

\subsection{A removal lemma}

Next we show how to deduce Theorem \ref{superstability} from the following stability theorem.
We say that a whitened igraph $G$ on $[n]$ is a \emph{$c$-whitened igraph} 
if it has at most $cn^2$ white pairs.

\begin{thm}\thlabel{stability}
For every $\aA>0$, $a,b \ge 2$, and $p \in (0,1)$ there exists $n_0$ and $\bB>0$, 
such that for any $I_{a,b}$-free $\bB$-whitened igraph $G$ on $n > n_0$ vertices 
with $w_p(G) > (\kK_p(I_{a,b}) - \bB)\tbinom{n}{2}$
we can modify at most $\aA n^2$ pairs of $G$ to obtain $G' \in \mc{G}_p(n)$.
\end{thm}
 
We also require the following igraph removal lemma; we omit the proof, as it is very similar 
to that of the usual graph removal lemma of Erd\H{o}s, Frankl and R\"odl \cite{EFR}.

\begin{lemma}\thlabel{removal}
For any igraph $H$ and $\aA>0$ there is $\bB>0$ such that 
any igraph on $n$ vertices with at most $\bB n^{|V(H)|}$ copies of $H$ 
can be made $H$-free by whitening at most $\aA n^2$ pairs. 
\end{lemma}

\nim{Proof of Theorem \ref{superstability}}.
Let $0<\dD<\aA<\eps$ and $n_0$ be chosen so that Lemma \ref{removal} can be applied with $H=I_{a,b}$ and $\bB=\dD$,
and Theorem \ref{stability} can be applied with $(\eps,3\aA)$ in place of $(\aA,\bB)$.
By Lemma \ref{removal}, by whitening at most $\aA n^2$ pairs of $G$ we obtain $G^*$ that is $I_{a,b}$-free.
Note that $w_p(G^*) > w_p(G) - \aA n^2 > (\kK_p(I_{a,b}) - 3\aA)\tbinom{n}{2}$.
Then by Theorem \ref{stability} we can modify at most $\eps n^2$ pairs
to obtain $G' \in B_a(n) \cup R_{b-1}(n)$. \qed

\subsection{Types} \label{types}

For the proof of Theorem \ref{stability} we require some of the theory
developed by Marchant and Thomason \cite{MT}, which we will now describe.
We start with some definitions. 
A \emph{type} $\tau$ is an igraph in which every vertex is coloured red or blue.
We let $m \times \tau$ denote the igraph $G$ that is the `$m$-fold blowup of $\tau$':
$V(G)$ is partitioned into parts $(V_u: u \in V(\tau))$ of size $m$
such that all edges inside a part $V_u$ have the colour of $u$ in $\tau$,
and all edges between two parts $V_u$ and $V_v$ have the colour of $uv$ in $\tau$
(in the three-coloured viewpoint).
We say that an igraph $G$ is \emph{described} by $\tau$
if it is a subgraph of $m \times \tau$ for some $m$.

One of the main results of Marchant and Thomason is that
for any hereditary property $\mc{P} = \{ \mc{H}-\text{free graphs} \}$ 
and $p \in (0,1)$ there is some type $\tau$ such that $\kK_p(\mc{H})$
is asymptotically achieved by igraphs described by $\tau$.
To calculate this value, we introduce the following notation.
Define $w_p(ij)$ to be $p$, $1-p$, $1$
according as $ij$ is red, blue, green,
and $w_p(ii)=w_p(i)$ to be $p$ or $1-p$
according as $i$ is red or blue.
We let $W_p(\tau)$ be the $V(\tau) \times V(\tau)$ 
symmetric matrix with $W_p(\tau)_{ij} = w_p(ij)$ for all $i,j$.
We also let $\DD(\tau)$ be the simplex of all $x = (x_u: u \in V(\tau))$
with $\sum_u x_u = 1$ and all $x_u \ge 0$. The \emph{$p$-value} of $\tau$ is
\[ \lL_p(\tau) = \max_{x \in \DD(\tau)} x^T W_p(\tau) x.\]
By considering igraphs $G$ on $n$ vertices described by $\tau$
with $||V_u| - x_u n| \le 1$ for $u \in V(\tau)$
we obtain $p$-weights $w_p(G) \sim \lL_p(\tau) \tbinom{n}{2}$.
Thus $\kK_p(\mc{H})$ is the maximum of $\lL_p(\tau)$
over all types $\tau$ that do not describe any $H \in \mc{H}$.

We say that a type $\tau$ is \emph{$p$-core}
if no proper subtype of $\tau$ has the same $p$-value.
Note that $\kK_p(\mc{H})$ is the maximum of $\lL_p(\tau)$
over all $p$-core types $\tau$ that do not describe any $H \in \mc{H}$.
This observation is useful because $p$-core types
have the following special structure.

\begin{thm}\thlabel{coreType} (\cite[Theorem 3.23]{MT})
Let $\tau$ be a $p$-core type. Then all edges of $\tau$ are green, apart from
	\begin{itemize}[nolistsep]
		\item if $p < 1/2$ when some edges between two red vertices can be blue, or
		\item if $p > 1/2$ when some edges between two blue vertices can be red.
	\end{itemize}
\end{thm}

One can easily deduce (\ref{eqk}). Indeed, note that $I_{a,b}$ is described by the type 
with one blue vertex and one red vertex joined by a green edge, 
so Theorem \ref{coreType} has the following consequence.

\begin{fact}\thlabel{typemono}
Let $\tau$ be a $p$-core type that does not describe $I_{a,b}$.
Then all vertices of $\tau$ have the same colour.
\end{fact}

Let $\tau(x,y)$ denote the type with $x$ red vertices and $y$ blue vertices 
in which all edges are green. Note that $\tau(0,y)$ describes $I_{a,b}$ for $y>a$
and $\tau(x,0)$ describes $I_{a,b}$ for $x \ge b$. Thus $\kK_p(I_{a,b})$ 
is the maximum of $\lL_p(\tau(0,a)) = 1-\tfrac{p}{a}$ and $\lL_p(\tau(b-1,0)) = 1-\tfrac{1-p}{b-1}$.
Note that these values are equal when $p = \tfrac{a}{a+b-1}$,
which indicates why this is the threshold parameter in Theorem \ref{exact}.

Later we will also need to know that these are the only extremal types (see Lemma \ref{uniquetype}).
First we introduce some more notation. Suppose $G$ is a whitened igraph, $x \in V(G)$ and $S \sub V(G)$.
We write $d_r(x,S)$ for the number of vertices $y \in S$ such that $xy$ is red, in the two-coloured viewpoint.
If $S = V(G)$ we omit it from the notation. 
We define $d_b(x,S)$, $d_g(x,S)$ and $d_w(x,S)$ similarly for blue, green and white.
For $p \in (0,1)$, the \emph{$p$-degree} of $x$ in $S$ is $d_p(x,S) = pd_r(x) + (1-p)d_b(x)$.
Note that $d_r(x)+d_b(x)+d_w(x) = d_g(x) + |V(G)|-1$, $\sum_{x \in V(G)} d_r(x) = 2|G_r|$, 
$\sum_{x \in V(G)} d_b(x) = 2|G_b|$, $\sum_{x \in V(G)} d_g(x) = 2|G_g|$
and $\sum_{x \in V(G)} d_p(x) = 2w_p(G)$.
The \emph{minimum $p$-degree} of $G$ is $\dD_p(G) = \min_{x \in V(G)} d_p(x)$.
The following standard lemma shows that we can ensure large minimum $p$-degree by deleting a few vertices.

\begin{lemma}\thlabel{minDens}
For any graph $H$ and $\dD>0$ there exists $n_0$ such that
for any $H$-free whitened igraph $G$ on $n>n_0$ vertices with $w_p(G) > (\kK_p(H)-\dD) \tbinom{n}{2}$,
there is $S \sub V(G)$ with $|S| \le \sqrt{\dD} n$ such that $\dD_p(G \sm S) > (\kK_p(H)-\sqrt{\dD}) |V(G \sm S)|$.
\end{lemma}

\begin{proof}
Suppose that we have a sequence $G=G_n, G_{n-1}, \dots, G_{n-\sqrt{\dD} n} = G'$ where for each $i$
we obtain $G_{i-1}$ from $G_i$ by deleting a vertex of $p$-degree at most $(\kK_p(H)-\sqrt{\dD}) i$. Then
\begin{align*}
w_p(G') & \ge w_p(G) - \sum_{i = n-\sqrt{\dD}n+1}^n (\kK_p(H)-\sqrt{\dD}) i \\
& \ge (\kK_p(H)-\dD) \tbinom{n}{2} - (\kK_p(H)-\sqrt{\dD}) \brac{ \tbinom{n+1}{2} - \tbinom{n-\sqrt{\dD}n+1}{2} } \\
& \ge (\sqrt{\dD}-\dD) \tbinom{n}{2} - n + (\kK_p(H)-\sqrt{\dD}) \tbinom{n-\sqrt{\dD}n}{2}.
\end{align*}
For large $n$ we have
\[(\dD/2+\sqrt{\dD}) \tbinom{n-\sqrt{\dD}n}{2} < (\dD/2+\sqrt{\dD}) (1-\sqrt{\dD})^2 n^2/2 
= (\sqrt{\dD} - 3\dD/2 + \dD^2/2) n^2/2 < (\sqrt{\dD}-\dD) \tbinom{n}{2} - n\]
so $w_p(G') > (\kK_p(H)+\dD/2) \tbinom{n-\sqrt{\dD}n}{2}$.
This contradicts the definition of $\kK_p(H)$ for large $n$,
so the sequence must terminate before we reach $G'$, which gives $S$ as required.
\end{proof}

Next we note another easy consequence of Theorem \ref{coreType}.

\begin{fact}\thlabel{typenb}
Let $\tau$ be a $p$-core type that does not describe $I_{a,b}$.
\begin{itemize}[nolistsep]
	\item If $\tau$ has blue vertices then $|V(\tau)| \le a$.
	\item If $\tau$ has red vertices then $d_g(x)<b-1$ for all $x \in V(\tau)$. 
\end{itemize}
\end{fact}

A variational argument gives the following property 
of any vector achieving the $p$-value of a type.
For a convenient statement we define $w_p(uu) = w_p(u)$.

\begin{fact}\thlabel{vary} (\cite[Fact 3.19]{MT})
Suppose $x \in \DD(\tau)$ with $\lL_p(\tau) = x^T W_p(\tau) x$.
Then for any $u \in V(\tau)$ we have
\[\lL_p(\tau)  = \sum_{v \in V(\tau)} w_p(uv) x_v.\]
\end{fact}

Now we deduce the following lemma; the proof is similar to that of \cite[Lemma 5.4]{MT}.

\begin{lemma}\thlabel{uniquetype}
For $a,b \ge 2$, any $p$-core type of maximum $p$-value 
not describing $I_{a,b}$ is $\tau(0,a)$ or $\tau(b-1,0)$. 
\end{lemma}

\begin{proof}
Consider any $p$-core type $\tau$ not describing $I_{a,b}$.
By Fact \ref{typemono} all its vertices have the same colour.
Suppose first that all vertices are blue.
By Fact \ref{typenb} we have $|V(\tau)| \le a$,
so the maximum $p$-value is achieved when $\tau = \tau(0,a)$.
Now suppose that all vertices are red.
By Theorem \ref{coreType} all edges of $\tau$ are blue or green.
If $\tau$ has at most $b-1$ vertices then the maximum $p$-value is achieved when $\tau = \tau(b-1,0)$.
Now suppose that $\tau$ has $k \ge b$ vertices.
Fix $x \in \DD(\tau)$ such that $\lL_p(\tau) = x^T W_p(\tau) x$.
By Fact \ref{vary}, for any $u \in V(\tau)$ we have
\[\lL_p(\tau)  = px_u + (1-p)(1-x_u) + p\sum_{v: uv \in E(G_g)} x_v.\]
Summing over $u$ we we obtain 
$k\lL_p(\tau)  = p + (1-p)(k-1) + p\sum_v x_v d_g(v)$.
By Fact \ref{typenb} we have $d_g(v) \le b-2$ for all $v$,
so $k\lL_p(\tau) \le (b-1)p+(1-p)(k-1)$,
giving $\lL_p(\tau) \le 1-p + (bp-1)/k$.
Finally, we consider two cases according to the value of $p$.
If $p \le 1/b$ we have $bp-1 \le 0$, so as $a \ge 2$,
we have $\lL_p(\tau) \le 1-p < 1-\tfrac{p}{a} = \lL_p(\tau(0,a))$.
If $p > 1/b$ we have $bp-1 > 0$,
so $\lL_p(\tau) \le 1-p + (bp-1)/b = 1-1/b < 1-\tfrac{1-p}{b-1} = \lL_p(\tau(b-1,0))$.
\end{proof}

Note that Lemma \ref{uniquetype} is false for $a=1$ when $p = 1/b$ (see the concluding remarks).


\subsection{Extensions} \label{extensions}

For the first step in the proof of Theorem \ref{stability}
we will need a method for extending types.
We say that a type $\sS$ is an \emph{extension} of $\tau$
if we can delete a vertex of $\sS$ to obtain $\tau$.
The proof of the Lemma \ref{extend} is heavily based on \cite[Lemma 3.11]{MT},
but is sufficiently different that we provide the details for the convenience of the reader.
It also uses the following two lemmas.

\begin{lemma}(\cite[Lemma 3.13]{MT}) \label{MT3.13}
Let $\tau$ be a type and let $x \in \Delta$ 
satisfy $\lambda_p(\tau) =  x^T W_p(\tau) x$.
Let $\sigma$ be an extension of $\tau$ such that $\sum_{u \in V(\tau)} x_u w_p(uv) > \lL_p(\tau)$, 
where $V(\tau) = V(\sS) \sm \{v\}$. Then $\lambda_p(\sigma) > \lambda_p(\tau)$.

\end{lemma}

\begin{lemma}(\cite[Lemma 3.22]{MT}) \label{MT3.22}
Let $\tau$ be a $p$-core type where $p \leq 1/2$ and let 
$x \in \Delta( \tau)$ satisfy $x^T W_p(\tau) x= \lambda_p(\tau) $.
Let $B$ be the set of blue vertices of $\tau$, 
and suppose that all edges of $\tau$ incident with B are green.
Let $\sigma$ be an extension of $\tau$ such that 
$\sum_{u \in V(\tau)} x_u w_p(uv) > \lL_p(\tau)$, 
where $V(\tau) = V(\sS) \sm \{v\}$. 

Suppose now that $\tau'$ is a $p$-core sub-type of $\sigma$ 
with $\lambda_p( \tau') = \lambda_p( \sigma)$. 
Then $B \subset V(\tau' )$, and all edges of $\tau'$ 
incident with $B$ are green. The analogous statement holds
when $p \geq 1/2$ replacing `red' with `blue'.
\end{lemma}

\begin{lemma}\thlabel{extend}
For any $p,d \in (0,1)$, integer $m$ and type $\tau$ with $\lL_p(\tau) < d$
there are $\bB>0$ and integers $M$, $n_0$ such that 
for any $\bB$-whitened igraph $G$
on $n>n_0$ vertices with $\dD_p(G) > dn$, 
writing $V_w = \{x: d_w(x) \ge \sqrt{\bB} n \}$,
if $G \sm V_w$ contains a copy of $M \times \tau$ 
then $G \sm V_w$ contains $m \times \sS$ 
for some extension $\sS$ of $\tau$ with $\lL_p(\sS)>\lL_p(\tau)$.

Furthermore, if $p \le 1/2$, writing $B$ for the set of blue vertices of $\tau$,
if all edges of $\tau$ incident with $B$ are green and $\tau'$ is a $p$-core subtype of $\sS$
with $\lL_p(\tau') = \lL_p(\sS)$ then $B \sub V(\tau')$
and all edges of $\tau'$ incident with $B$ are green.
For $p \ge 1/2$ the analogous statement holds with `red' instead of `blue'.
\end{lemma}

\begin{proof}
Fix $x \in \DD(\tau)$ with $x^T W_p(\tau) x = \lL_p(\tau)$. 
By Lemmas \ref{MT3.13} and \ref{MT3.22}, it suffices to find 
$m \times \sS$ for some extension $\sS$ of $\tau$ such that 
$\sum_{u \in V(\tau)} x_u w_p(uv) > \lL_p(\tau)$, 
where $V(\tau) = V(\sS) \sm \{v\}$.
We choose the following parameters. 
Choose $\eta>0$ so that $5|V(\tau)|\eta < d-\lL_p(\tau)$
and $\eta < x_u/6$ for any $u \in V(\tau)$ with $x_u \ne 0$. 
Choose $M > 6\eta^{-1}m|V(\tau)|$.
Choose $\bB>0$ so that $|V(\tau)| \sqrt{\bB} < \eta^2/2$,
$\bB^{1/3} < \brac{3\tbinom{M}{m}}^{-|V(\tau)|} \eta/3$
and $\bB^{1/3} < 1/3R(m)$, where $R(m)$ is the two-colour 
Ramsey number for $K_m$.
Choose $n_0 > M/9\eta$.

Now let the copy of $M \times \tau$ in $G$ have vertex classes $(W_u: u \in V(\tau))$ and let $W = \cup_u W_u$.
For each $u$ fix $X_u \sub W_u$ with $|X_u| = \bcl{x_u M}$ 
and let $X = \cup_u X_u$.
Let $Y$ be the set of $z \in V(G) \sm X$ 
such that $d_p(z,X) > (d-\eta)|X|$. We have
\[ dn|X| < \sum_{y \in X} d_p(y) = \sum_{z \in V(G)} d_p(z,X)
\le (|X|+|Y|)|X| + n(d-\eta)|X|.\]
This gives $|X|+|Y| \ge \eta n$, so $|Y| \ge \eta n - M$.
Let $Y_0$ be the set of $z \in Y$ such that $d_w(z,W_u) > \eta M$ for some $u \in V(\tau)$. We have
\[ |Y_0| \eta M < \sum_{z \in Y_0} d_w(z,W) = \sum_{x \in W} d_w(x,Y_0) < |V(\tau)| M \sqrt{\bB} n.\]
We deduce that $|Y_0| < \eta n/2$. We also note that 
\[ |V_w| \sqrt{\bB}n \le \sum_{y \in V_w} d_w(y) \le 2|G_w| < 2\bB n^2, \]
so $|V_w| < 2 \sqrt{\bB}n$. Then letting $Y' = Y \sm (Y_0 \cup V_w)$,
we have $|Y'| > \eta n - M - \eta n/2 - 2 \sqrt{\bB}n > \eta n/3$.

Now fix any $y \in Y'$. We claim that for each $u \in V(\tau)$ 
we can choose $U_u \sub W_u$ with $|U_u|=m$, such that the edges 
between $y$ and $U_u$ are all the same colour, which is not white.
Writing $w_u$ for the weight of each edge between $y$ and $U_u$
(which is well-defined as they all have the same colour),
we claim moreover that if $x_u \ne 0$ we can choose $U_u$ 
so that $w_u \ge d_p(y,X_u)/|X_u| - 3\eta/x_u$.
To see these claims, suppose first that $x_u=0$. 
There are at least $(1-\eta)M$ vertices $a \in W_u$
such that $ay$ is not white (as $y \notin Y_0$), 
so at least $(1-\eta)M/3 \ge m$ 
have the same colour, as required. 
Now suppose that $x_u>0$ but the claims fail.
Then there are at most $3m$ coloured edges and at most $\eta M$ white edges
between $y$ and $W_u$ that satisfy the required bound for $w_u$. 
By choice of $\eta$ we have $3m < \eta M$ 
and $x_u-2\eta < 2x_u/3$, so we obtain 
\begin{align*}
 d_p(y,X_u) & <  2\eta M + (|X_u| - 2\eta M) (d_p(y,X_u)/|X_u| - 3\eta/x_u) \\
& \le 2\eta M + d_p(y,X_u) - (|X_u|- 2\eta M) (3\eta/x_u) \\
& \le 2\eta M + d_p(y,X_u) - (2	x_u M/3) (3\eta/x_u) \le d_p(y,X_u).
\end{align*}
This contradiction establishes the claims. An immediate consequence is that
\[ \sum_u x_u w_u \ge \sum_u (d_p(y,X_u)/M-3\eta) 
> d-5|V(\tau)|\eta > \lL_p(\tau).\]
Let $U(y) = \cup_u U_u$. We choose $Z \sub Y'$
such that $U(y)$ is the same set for all $y \in Z$
and the edges between $y$ and $U(y)$ are coloured 
identically for all $y \in Z$;
by the pigeonhole principle we can take 
$|Z| \ge \brac{3\tbinom{M}{m}}^{-|V(\tau)|} |Y'| > \bB^{1/3} n$. 
Since $|G_w| < \bB n^2$, by Tur\'an's theorem we can find $S \sub Z$ 
that does not contain any white edges with $|S| > \bB^{-1/3}/3 > R(m)$.
Then by Ramsey's theorem we can find $U_v \sub S$ of size $m$
in which all edges are red or all edges are blue (counting green as either).
Now $(U_u: u \in V(\sS))$ is a copy of $m \times \sS$ in $G \sm V_w$
with $\sum_u x_u w_p(uv) = \sum_u x_u w_u > \lL_p(\tau)$,
so $\lL_p(\sS) > \lL_p(\tau)$.
\end{proof}

By iterating the previous lemma, in the following lemma we deduce that 
we can find a blowup of a $p$-core type that either has large $p$-value 
or a large number of vertices.

\begin{lemma}\thlabel{extendlarge}
For any $p,d \in (0,1)$ and integers $s,t$ there are $\bB>0$ and $n_0$ 
such that for any $\bB$-whitened igraph $G$ on $n>n_0$ vertices with $\dD_p(G) > dn$, 
writing $V_w = \{x: d_w(x) \ge \sqrt{\bB} n \}$, $G \sm V_w$ contains $s \times \tau$ 
for some $p$-core type $\tau$ with $\lL_p(\tau) \ge d$ or $|V(\tau)| > t$.

Furthermore, there is some integer $S$ such that if $p \le 1/2$ and $G \sm V_w$ contains 
a blue clique of size $S$ then we can ensure that $\tau$ has at least one blue vertex. 
For $p \ge 1/2$ the analogous statement holds with `red' instead of `blue'. 
\end{lemma}

\begin{proof}
Let $z$ be the number of types with at most $t$ vertices.
We define $s_0,\dots,s_z$, $\bB_1,\dots,\bB_z$ and $n_1,\dots,n_z$ inductively as follows:
let $s_0=s$, and for $i=1,\dots,z$ let $s_i$, $\bB_i$, $n_i$ be such that
Lemma \ref{extend} can be applied with $m=s_{i-1}$, $M=s_i$, $\bB=\bB_i$, $n_0=n_i$
and any type $\tau$ on at most $t$ vertices with $\lL_p(\tau)<d$. 
Let $n_0 = \sum_{i=1}^z n_i$ and $\bB = \prod_{i=1}^z \bB_i$. 
Now consider $G$ as in the statement of the lemma.
By Lemma \ref{extend} applied with $\tau=\es$ we obtain $s_z \times \tau_z$
for some type $\tau_z$ with one vertex. 
Then we construct a sequence $s_i \times \tau_i$ in $G \sm V_w$ for $i=z,z-1,\dots$,
where if $\lL_p(\tau_i)<d$ and $|V(\tau_i)| \le t$ we let $\sS$ be the type
provided by Lemma \ref{extend} applied to $\tau = \tau_i$ 
and let $\tau_{i-1}$ 
be a $p$-core subtype of $\sS$ with 
$\lL_p(\tau_{i-1}) = \lL_p(\sS) > \lL_p(\tau_i)$. 
This sequence must terminate at some $i>0$; 
then $\tau=\tau_i$ satisfies the lemma.
For the `furthermore' statement, we take $S=s_z$
and the type $\tau_z$ to be a blue vertex;
then the type $\tau$ constructed above
has the required properties by Theorem \ref{coreType}
and the `furthermore' statement of Lemma \ref{extend}.
\end{proof}

\section{Stability}

We prove Theorem \ref{stability} in two steps.
The first step is to find a copy of either $t \times \tau(0,a)$ or $t \times \tau(b-1,0)$ in $G$,
for some $t$ that is a large constant but small compared with $n$.
The second step is to study how the rest of the igraph interacts with this structure:
we will see that it extends to all but a small proportion of the vertices of $G$.
The following lemma accomplishes the first step.

\begin{lemma}\thlabel{stabpart1}
For every $a,b \ge 2$, $p \in (0,1)$ and integer $t$ there exists $n_0$ and $\bB>0$, 
such that for any $I_{a,b}$-free $\bB$-whitened igraph $G$ on $n > n_0$ vertices 
with $\dD_p(G) > (\kK_p(I_{a,b}) - \bB)n$, writing $V_w = \{x: d_w(x) \ge \sqrt{\bB} n \}$, 
\begin{itemize}[nolistsep]
\item if $p<\tfrac{a}{a+b-1}$ then $G \sm V_w$ contains $t \times \tau(0,a)$,
\item if $p>\tfrac{a}{a+b-1}$ then $G \sm V_w$ contains $t \times \tau(b-1,0)$,
\item if $p=\tfrac{a}{a+b-1}$ then $G \sm V_w$ contains $t \times \tau(0,a)$ or $t \times \tau(b-1,0)$.
\end{itemize}
\end{lemma}

\begin{proof}
We can assume $t \ge a+b$.
Then if $G$ contains $ t \times \tau$ for some type $\tau$ 
then $\tau$ does not describe $I_{a,b}$. 
Since there are finitely many types $\tau$ on at most $S$ vertices, 
there is some $\bB_3$ such that no such $\tau$ has 
$\kK_p(I_{a,b})-\bB_3 \le \lL_p(\tau) < \kK_p(I_{a,b})$.
Let $\bB_1$, $n_1$, $S$ be such that Lemma \ref{extendlarge} can be applied with $d=\kK_p(I_{a,b}) - \bB_3$ and $s=t$.
Let $\bB_2$, $n_2$ be such that Lemma \ref{extendlarge} can be applied with $d=\kK_p(I_{a,b}) - \bB_3$ and $s=t=S$.
Let $\bB = \bB_1 \bB_2 \bB_3$ and $n_0 = n_1+n_2+n_3$.
Now consider $G$ as in the statement of the lemma. 
By choice of $\bB$ and Lemma \ref{uniquetype},
it suffices to show that $G \sm V_w$ contains $t \times \tau$ 
for some type $\tau$ 
with $|V(\tau)| \le S$ and $\lL_p(\tau) \ge \kK_p(I_{a,b}) - \bB_3$.

Suppose first that $p \le 1/2$. By Lemma \ref{extendlarge}, 
$G \sm V_w$ contains $S \times \tau$ 
for some $p$-core type $\tau$ 
with $\lL_p(\tau) \ge \kK_p(I_{a,b}) - \bB_3$ or $|V(\tau)| > S$.
We can assume that $|V(\tau)| > S$, otherwise we are done.
By Fact \ref{typemono} all vertices of $\tau$ have the same colour.
Next we show that by replacing $S$ by $t$
we can assume this colour is blue.
Indeed, if it is red, 
then by Fact \ref{coreType} all edges of $\tau$ are blue or green,
so there is a blue clique of size $S$ in $\tau$, and so in $G$. 
Then we apply Lemma \ref{extendlarge}
again to see that in any case $G \sm V_w$ contains $t \times \tau$ 
for some $p$-core type $\tau$ with blue vertices such that 
$\lL_p(\tau) \ge \kK_p(I_{a,b}) - \bB_3$ or $|V(\tau)| > t$.
Since all edges of $\tau$ are green by Fact \ref{coreType},
we cannot have $|V(\tau)| > t$, as $t \times \tau(0,t)$ contains $I_{a,b}$.
Then $\lL_p(\tau) \ge \kK_p(I_{a,b}) - \bB_3$, so we are done.
The argument for $p \ge 1/2$ is the same, interchanging `red' and `blue'.
\end{proof}

For the second step of the proof of Theorem \ref{stability}, we use double-counting 
to extend a large blowup of an extremal graph to almost all the vertex set.
The next lemma accomplishes this for the case $p \le \tfrac{a}{a+b-1}$.
Before stating it we record some observations on $I_{a,b}$ (we omit the easy proof).

\begin{fact} \label{findIa}
Let $G$ be an $I_{a,b}$-free whitened igraph containing a copy $P$ of $t \times \tau(0,a)$ with parts $P_1,\dots,P_a$.
\begin{itemize}[nolistsep]
\item There is no $z \in V(G) \sm P$ with $d_r(z,P_i) \ge 1$ for all $i$, 
and $d_r(z,P_j) \ge b$ for some $j$.
\item There is no red edge $zz'$ disjoint from $P$ such that $z$ and $z'$
have at least $b$ common red neighbours in some $P_j$,
and at least $1$ common red neighbour in all but one other $P_i$.
\end{itemize}
\end{fact}

\begin{lemma}\thlabel{stabpart2a}
For every $a,b \ge 2$, $0 < p \le \tfrac{a}{a+b-1}$ and $\aA>0$ there exists $t$ and $\bB>0$, 
such that for any $I_{a,b}$-free $\bB$-whitened igraph $G$ on $n$ vertices with 
$\dD_p(G) > (\kK_p(I_{a,b}) - \bB)n$ and $t \times \tau(0,a) \sub G$, 
there is $X \sub V(G)$ with $|X|<\aA n$ such that $G \sm X$ can be partitioned 
into $a$ sets each containing no red edge. 
\end{lemma}

\begin{proof}
Note that $\kK_p(I_{a,b})=1-p/a$. We can assume $\aA<1/4$. We choose $t>3b$ and $\bB < pa^{-1}\aA^2$.
Let $P_1,\dots,P_a$ be the parts of $t \times \tau(0,a)$ and let $P = \cup_{i=1}^a P_i$.
We claim that for any $z \in V(G) \sm P$ we have $d_p(z,P) \le (a-p)t$.
To see this, suppose for a contradiction that $d_p(z,P) > (a-p)t$.
Since $d_p(z,P) \le d_r(z,P) + (1-p)(at-d_r(z,P)) = (1-p)at + pd_r(z,P)$ we have $d_r(z,P) > (a-1)t$.
We deduce that $d_r(z,P_i) \ge 1$ for all $i$, and there is some $i$ with $d_r(z,P_i) \ge (a-1)t/a \ge b$.
However, this is a contradiction by Fact \ref{findIa}, so we have $d_p(z,P) \le (a-p)t$, as claimed.

Let $X = \{z \in V(G) \sm P: d_p(z,P) < (a-p-\aA p)t \}$. By double-counting and the claim we have
\[ (1-p/a-\bB)n at < \sum_{y \in P} d_p(y) = \sum_{z \in V(G)} d_p(z,P)
\le (n-|X|) (a-p)t + |X| (a-p-\aA p)t.\]
This gives $|X| \aA pt \le \bB n at $, so $|X| \le a\bB n/p\aA < \aA n$.

Next we claim that for any $y \in V(G) \sm (P \cup X)$ there is some part $P_j$ 
such that $d_r(y,P_j) = 0$ and $d_r(y,P_i) > 2t/3$ for all $i \ne j$.
To see this, note first that there must be some $i$ with $d_r(y,P_i) \ge b$, otherwise, as $a \ge 2$, 
we have $d_p(y,P) \le (1-p)at + pab \le at - 2p(t-b) < (a-p-\aA p)t$, contradicting the definition of $X$.
Then by Fact \ref{findIa} there is some $P_j$ such that $d_r(y,P_j) = 0$.
We deduce $d_p(y,P) \le (1-p)t + d_r(y,P) + (1-p)((a-1)t-d_r(y,P))$,
so $pd_r(y,P) \ge (a-p-\aA p)t - (1-p)at = (a-1-\aA)pt$.
Thus for every $i \ne j$ we have $d_r(y,P_i) \ge (1-\aA)t > 2t/3$, as claimed.

Finally, we partition $V(G) \sm X$ as $(A_1,\dots,A_a)$, where $A_j = \{y \in V(G) \sm X: d_r(y,P_j)=0\}$.
This partition satisfies the lemma, as any two vertices in $A_i$ have at least $t/3 > b$ common red neighbours 
in any $P_j$ with $j \ne i$, so by Fact \ref{findIa} they cannot form a red edge.
\end{proof}

Next we prove a lemma that accomplishes the second step of the 
proof of Theorem \ref{stability}
in the case $p \ge \tfrac{a}{a+b-1}$. 
Before stating it we again record some easy observations on $I_{a,b}$
(note that we include the `$+1$' in the bound for $d_r$ 
as the same edge can be both red and blue,
but we need to choose $a+1$ distinct vertices in $P_j$).

\begin{fact} \label{findIb}
Let $G$ be an $I_{a,b}$-free whitened igraph containing a copy $P$ 
of $t \times \tau(b-1,0)$ with parts $P_1,\dots,P_{b-1}$.
\begin{itemize}[nolistsep]
\item There is no $z \in V(G) \sm P$ with $d_b(z,P_i) \ge 1$ for all $i$, 
and $d_r(z,P_j) \ge a+1$ for some $j$.
\item There is no blue edge $zz'$ disjoint from $P$ such that $z$ and $z'$
have at least $a+1$ common red neighbours and $1$ common blue neighbour in some $P_j$,
and at least $1$ common blue neighbour in all but one other $P_i$.
\end{itemize}
\end{fact}

\begin{lemma}\thlabel{stabpart2b}
For every $a,b \ge 2$, $\tfrac{a}{a+b-1} \le p < 1$ and $\aA>0$ there exists $t$ and $\bB>0$, 
such that for any $I_{a,b}$-free $\bB$-whitened igraph $G$ on $n$ vertices with 
$\dD_p(G) > (\kK_p(I_{a,b}) - \bB)n$ and $t \times \tau(b-1,0) \sub G$, 
there is $X \sub V(G)$ with $|X|<\aA n$ such that $G \sm X$ can be partitioned 
into $b-1$ sets each containing no blue edge. 
\end{lemma}

\begin{proof}
The proof is similar but slightly different to that of Lemma \ref{stabpart2a}.
Note that $\kK_p(I_{a,b})=1-\tfrac{1-p}{b-1}$. We can assume $\aA<1/6b$. 
We choose $t>12ab$ and $\bB < (1-p)(b-1)^{-1}\aA^2$.
Let $P_1,\dots,P_{b-1}$ be the parts of $t \times \tau(b-1,0)$ 
and let $P = \cup_{i=1}^{b-1} P_i$.

We claim that for any $z \in V(G) \sm P$ we have $d_p(z,P) \le (b-2+p)t$.
To see this, suppose for a contradiction that $d_p(z,P) > (b-2+p)t$.
First we consider the case $p < 1/2$. Note that
\[(b-2+p)t < d_p(z,P) \le d_g(z,P) + (1-p)((b-1)t-d_g(z,P)) = 
(1-p)(b-1)t + p d_g(z,P),\] 
so $bt - t/p \le d_g(z,P)$. 
As $\tfrac{a}{a+b-1} \le p < 1/2$ (so $b>a+1$)
and $a \ge 2$ we have
\[d_g(z,P) \ge t(b - \tfrac{a +b - 1}{a}) 
= (a-2)bt/a + (b-a+1)t/a > t.\]
We deduce that $d_b(z,P_i) \ge 1$ for all $i$, 
and there is some $i$ with $d_r(z,P_i) > \tfrac{t}{b-1} > a$.
However, this is a contradiction by Fact \ref{findIb}.
Now suppose $p \ge 1/2$. Then
\[(b-2+p)t < d_p(z,P) \le d_g(z,P) + p((b-1)t-d_g(z,P)) 
= p(b-1)t + (1-p)d_g(z,P),\] 
so $d_g(z,P) > (b-2)t$.
This contradicts Fact \ref{findIb} as in the first case,
unless $b=2$ and $d_g(z,P) \le a$. 
Then $d_b(z,P) \ge d_g(z,P) \ge 1$ and 
\[ pt < d_p(z,P) \le a + p(d_r(z,P)-a) + (1-p)(t-d_r(z,P)) 
= (1-p)(t+a) + (2p-1) d_r(z,P),\] 
so $d_r(z,P)>t-a>a$, which contradicts Fact \ref{findIb}.
The claim follows.

Let $X = \{z \in V(G) \sm P: d_p(z,P) < (b-2+p-\aA(1-p))t \}$. 
By double-counting and the claim we have
\[ (1-\tfrac{1-p}{b-1}-\bB)n (b-1)t < \sum_{y \in P} d_p(y) = \sum_{z \in V(G)} d_p(z,P)
\le (n-|X|) (b-2+p)t + |X| (b-2+p-\aA(1-p))t.\]
This gives $|X| \aA(1-p)t \le \bB n (b-1)t $, 
so $|X| \le (b-1)\bB n/(1-p)\aA < \aA n$.

Next we claim that for any $y \in V(G) \sm (P \cup X)$ there is some part $P_j$ 
such that $d_b(y,P_j) = 0$ and $d_g(y,P_i) > 2t/3$ for all $i \ne j$.
First we will show that there is some $i$ with $d_r(y,P_i) \ge a+1$.
Indeed, otherwise, by definition of $X$, we would have
\[ (b-2+p-\aA(1-p))t \le d_p(y,P) \le (1-p)(b-1)t + p(b-1)(a+1),  \]
i.e.\ $(bp-1)t \le \aA (1-p) t + p(b-1)(a+1)$.
As $bp \ge \tfrac{ab}{a+b-1} \ge 4/3$ we obtain
$t/3 \le (1-p)t/6b + p(b-1)(a+1) < t/6 + 2ab$, 
which is a contradiction, 
so we can find $i$ as required.
Then by Fact \ref{findIb} there is some $P_j$ such that $d_b(y,P_j) = 0$.
We deduce $d_p(y,P) \le pt + d_b(y,P) + p((b-2)t-d_b(y,P))$,
so \[(1-p)d_b(y,P) \ge (b-2+p-\aA(1-p))t - p(b-1)t = (b-2-\aA)(1-p)t,\]
i.e.\ $yz$ is blue or green for all but at most $\aA t < t/6$ vertices $z \in P \sm P_j$.
Similarly, $d_p(y,P) \le pt + d_r(y,P) + (1-p)((b-2)t-d_r(y,P))$,
so \[pd_r(y,P) \ge (b-2+p-\aA(1-p))t - pt - (1-p)(b-2)t = ((b-2)p-\aA(1-p))t,\]
i.e.\ $yz$ is red or green for all but at most $\aA (1-p) t/p < t/6$ vertices $z \in P \sm P_j$.
Therefore $yz$ is green for all but at most $t/3$ vertices $z \in P \sm P_j$, which gives the claim.

Finally, we partition $V(G) \sm X$ as $(A_1,\dots,A_{b-1})$, where $A_j = \{y \in V(G) \sm X: d_b(y,P_j)=0\}$.
This partition satisfies the lemma, as any two vertices in $A_i$ have at least $t/3 > a$ common green neighbours 
in any $P_j$ with $j \ne i$, so by Fact \ref{findIb} they cannot form a blue edge.
\end{proof}

We deduce Theorem \ref{stability} by combining the previous lemmas.

\nim{Proof of Theorem \ref{stability}.}
Let $t$ and $\bB_1$ be such that we can apply Lemmas \ref{stabpart2a} and \ref{stabpart2b} 
with $\bB=\bB_1$ and $\aA/2$ in place of $\aA$.
Let $n_2$ and $\bB_2$ be such that we can apply Lemma \ref{stabpart1} with $n_0=n_2$ and $\bB=\bB_2$.
Let $\bB=(\aA\bB_1\bB_2/10)^2$ and $n_1$ be such that we can apply Lemma \ref{minDens} 
with $H=I_{a,b}$, $\dD=\bB$ and $n_0=n_1$. 
Let $n_0=n_1+n_2$ and $G$ be as in the statement of the theorem.
By Lemma \ref{minDens} we have $S \sub V(G)$ with $|S| \le \sqrt{\bB} n$ 
such that $\dD_p(G \sm S) > (\kK_p(I_{a,b})-\sqrt{\bB}) (n-|S|)$.
By Lemma \ref{stabpart1} we can find $t \times \tau \sub G \sm S$,
where $\tau$ is $\tau(0,a)$ if $p<\tfrac{a}{a+b-1}$, $\tau$ is $\tau(b-1,0)$ if $p>\tfrac{a}{a+b-1}$,
and $\tau$ is $\tau(0,a)$ or $\tau(b-1,0)$ if $p=\tfrac{a}{a+b-1}$.
By Lemmas \ref{stabpart2a} and \ref{stabpart2b} there is $X \sub V(G) \sm S$ with $|X|<\aA n/2$ such that 
the structure of $t \times \tau$ extends to $G \sm (S \cup X)$. By modifying the at most $\aA n^2$ pairs of $G$
that are incident to $S \cup X$ or are white we obtain $G' \in \mc{G}_p(n)$, as required. \qed

\section{The exact result}

In this section we apply the stability method to refine our stability result and obtain an exact result.
We give some preliminary calculations and partite Tur\'an results in the first subsection,
then deduce our second main result in the second subsection.

\subsection{Preliminaries}

We start with a formula for the number of edges in a complete multipartite graph.

\begin{fact} \label{turansize}
Let $G$ be a complete $t$-partite graph on $n$ vertices with parts $V_i$ of sizes $n/t + a_i$ 
for $1 \le i \le t$. Then $2e(G) = (1-1/t)n^2 - \sum_{i=1}^t a_i^2$.
\end{fact}

\nim{Proof.}
Note that $\sum_{i=1}^t a_i = 0$. Each vertex in $V_i$ has degree $n-|V_i|$, so 
\begin{equation}
2e(G) = \sum_{v \in V(G)} d(v) = \sum_{i=1}^t (n/t+a_i)(n-n/t-a_i) = (1-1/t)n^2 - \sum_{i=1}^t a_i^2. \tag*{$\Box$}
\end{equation}

We deduce the following formulae for the weights of the extremal igraphs.

\begin{fact} \label{extremalweights}
For some $0 \le C \le a/8$ and $0 \le C' \le (b-1)/8$ we have
\begin{align*} 
w_p(B_a(n)) & = (1-\tfrac{p}{a}) \tbinom{n}{2} + p(1-1/a)n/2 - C, \text{\ and \ } \\
w_p(R_{b-1}(n)) & = (1-\tfrac{1-p}{b-1}) \tbinom{n}{2} + (1-p)(1-1/(b-1))n/2 - C'.
\end{align*}
\end{fact}

\begin{proof}
Write $n = qa + r$ where $0 \le r \le a-1$.
Then $B_a(n)$ has $r$ parts of size $q+1 = n/a + (1-r/a)$ and $a-r$ parts of size $q = n/a - r/a$.
By Fact \ref{turansize} we have
\[w_p(B_a(n)) = (1-p) \tbinom{n}{2} + p((1-1/a)n^2 - r(1-r/a)^2-(a-r)(r/a)^2)/2.\]
Since $r(1-r/a)^2+(a-r)(r/a)^2 = r - r^2/a \le a/4$ we obtain the first statement.
We omit the similar calculation for the second statement.
\end{proof}

The above formulae imply the following comparisons between the weights of the extremal igraphs
(we omit the easy deduction).

\begin{fact} \label{compareweights}
For $p \in (0,1)$ and large $n$ we have
\begin{itemize}[nolistsep]
	\item $w_p(R_{b-1}(n)) > w_p(B_a(n))$ if $p > \tfrac{a}{a+b-1}$ or $p = \tfrac{a}{a+b-1}$ and $a < b-1$,
	\item $w_p(R_{b-1}(n)) < w_p(B_a(n))$ if $p < \tfrac{a}{a+b-1}$ or $p = \tfrac{a}{a+b-1}$ and $a > b-1$,
	\item $w_p(R_{b-1}(n)) = w_p(B_a(n))$ if $p = \tfrac{a}{a+b-1} = 1/2$.
\end{itemize}
\end{fact}

The following calculation will arise in a bound for the minimum $p$-degree
(note that the choice of constant `$3$' is not important).

\begin{fact} \label{difference}
We have 
\begin{align*}
w_p(R_{b-1}(n)) - w_p(R_{b-1}(n-1)) & \ge \kK_p(I_{a,b}) n - 3 \text{ if } p \ge \tfrac{a}{a+b-1}, \text{ and } \\
w_p(B_a(n)) - w_p(B_a(n-1)) & \ge \kK_p(I_{a,b}) n - 3 \text{ if } p \le \tfrac{a}{a+b-1}.
\end{align*}
\end{fact}

\begin{proof}
We can obtain $B_a(n-1)$ by deleting a vertex $x$ from a part with maximum size.
Then $w_p(B_a(n)) - w_p(B_a(n-1)) = d_p(x) = n-1 - p(\bcl{n/a}-1) \ge (1-p/a)n-2$,
so we have the required bound in this case.
A similar calculation applies for $w_p(R_{b-1}(n)) - w_p(R_{b-1}(n-1))$.
\end{proof}

Next we estimate the degrees of a vertex in the various colours, given its $p$-degree.
 
\begin{fact}\thlabel{colordeg}
Let $p \in (0,1)$ and $G$ be an igraph on $[n]$.
Suppose $x \in V(G)$ satisfies $d_p(x) \ge \kK_p(I_{a,b}) n - C$.
Then $d_r(x) \ge \tfrac{a-1}{a} n - \tfrac{C}{p}$
and $d_b(x) \ge \tfrac{b-2}{b-1} n -  \tfrac{C}{1-p}$.

Furthermore, if $p \le 1/2$ then $d_g(x) \ge \tfrac{a-1}{a} n - \tfrac{C}{p}$
and if $p \ge 1/2$ then $d_g(x) \ge \tfrac{b-2}{b-1} n -  \tfrac{C}{1-p}$.
\end{fact}

\begin{proof}
First we consider the case $p \le \tfrac{a}{a+b-1}$, when $\kK_p(I_{a,b})=1-\tfrac{p}{a}$.
The inequality on $d_r(x)$ follows from $(1-\tfrac{p}{a})n - C \le d_p(x) \le (1-p)n + pd_r(x)$.
If $p \le 1/2$ we have the same inequality with $d_g(x)$ in place of $d_r(x)$.
We also have $(1-\tfrac{p}{a})n - C \le d_p(x) \le pn + (1-p)d_b(x)$,
so $d_b(x) \ge (1 - \tfrac{p}{a(1-p)})n - \tfrac{C}{1-p}$.
Since $\tfrac{p}{a(1-p)} \leq \tfrac{1}{b-1}$ for $p \le \tfrac{a}{a+b-1}$
we obtain the inequality on $d_b(x)$. 
If $p \ge 1/2$ we have the same inequality with $d_g(x)$ in place of $d_b(x)$.

Now suppose that $p \ge \tfrac{a}{a+b-1}$, so $\kK_p(I_{a,b})=1-\tfrac{1-p}{b-1}$.
The inequality on $d_b(x)$ follows from $(1-\tfrac{1-p}{b-1})n - C \le d_p(x) \le pn + (1-p)d_b(x)$.
If $p \ge 1/2$ we have the same inequality with $d_g(x)$ in place of $d_b(x)$.
We also have $(1-\tfrac{1-p}{b-1})n - C \le d_p(x) \le (1-p)n + pd_r(x)$,
so $d_r(x) \ge (1 - \tfrac{1-p}{(b-1)p})n - \tfrac{C}{p}$.
Since $\tfrac{(1-p)}{(b-1)p} \leq \tfrac{1}{a}$ for $p \ge \tfrac{a}{a+b-1}$
we obtain the inequality on $d_r(x)$. 
If $p \le 1/2$ we have the same inequality with $d_g(x)$ in place of $d_r(x)$.
\end{proof}

We conclude this preliminary subsection with two results on multipartite Tur\'an problems.
The first is folklore; for a proof see e.g.\ \cite[Lemma 3.3]{BMSW}.

\begin{lemma} \label{partiteturan}
Suppose $G$ is a graph with $e(G) < m^2$ and $V(G)$ is partitioned into parts of size $m$.
Then we can find an independent set with one vertex in each part.
\end{lemma}

We deduce the following variant (it is not sharp, but suffices for our purposes).

\begin{lemma} \label{partiteturanb}
Suppose $G$ is a graph with $e(G) < (m/2t)^2$ and $V(G)$ is partitioned into parts of size $m>2t$.
Then we can find an independent set that has $t$ vertices in each part.
\end{lemma}

\begin{proof}
By Tur\'an's theorem, any set of $m/2 + 1$ vertices contains an independent set of size $t$.
Thus we can choose $\bcl{m/2t}$ vertex-disjoint independent sets of size $t$ inside each part.
Consider the auxiliary graph $H$ whose vertices correspond to the chosen independent sets,
where we join two vertices of $H$ if there is any edge of $G$ between the corresponding independent sets.
Then $e(H) \le e(G) < (m/2t)^2$, so by Lemma \ref{partiteturan}
we can find an independent set in $H$ with one vertex in each part.
This corresponds to an independent set in $G$ with $t$ vertices in each part.
\end{proof}

\subsection{Proof of Theorem \ref{exact}.}

Suppose $G$ is an $I_{a,b}$-free igraph on $[n]$ with maximum $p$-weight.
Note that $w_p(G) \ge w_p(R_{b-1}(n))$ and $w_p(G) \ge w_p(B_a(n))$.
By Fact \ref{compareweights} it suffices to show that 
$G$ is $R_{b-1}(n)$ or $B_a(n)$; if $w_p(R_{b-1}(n)) \ne w_p(B_a(n))$
it will follow that $G$ is the one of larger $p$-weight.
We claim that we can assume $\dD_p(G) \ge \kK_p(I_{a,b}) n - 4$.
For suppose we have proved the theorem under this assumption for $n \ge n_0$.
Consider $n > n_0^2$ and form igraphs $G = G_n, G_{n-1}, \dots$, 
where if $\dD_p(G_i) < \kK_p(I_{a,b}) i - 4$ 
we form $G_{i-1}$ by deleting a vertex of minimum $p$-degree from $G_i$.
We claim that this process must terminate at some $G_m$ with $m>n_0$.
Otherwise, we obtain $w_p(G_{n_0}) > w_p(G) - \sum_{i=n_0+1}^n (\kK_p(I_{a,b}) i - 4)$,
so by Fact \ref{difference} we have $w_p(G_{n_0}) > n-n_0 > \tbinom{n_0}{2}$,
which is a contradiction. Now $\dD_p(G_m) \ge \kK_p(I_{a,b})m - 4$, so by assumption
we have $w_p(G_m) \le w_p(R_{b-1}(m))$ if $p \ge \tfrac{a}{a+b-1}$
and $w_p(G_m) \le w_p(B_a(m))$ if $p \le \tfrac{a}{a+b-1}$.
If $p \le \tfrac{a}{a+b-1}$, by Fact \ref{difference} we have
\[w_p(G) \le w_p(B_a(m)) + \sum_{j=m+1}^n (w_p(B_a(j)) - w_p(B_a(j-1))-1) 
= w_p(B_a(n)) - (n-m),\] 
so we must have $n=m$ and $G=G_n=G_m$, so the theorem follows.
A similar calculation applies if $p \ge \tfrac{a}{a+b-1}$.
Therefore we can assume $\dD_p(G) \ge \kK_p(I_{a,b}) n - 4$.

Next, let $\eps = p(1-p)(10ab)^{-4}$ and $n$ be large. By Theorem \ref{stability}, 
we can modify at most $\eps^2 n^2/2$ pairs of $G$ to obtain $G' \in \mc{G}_p(n)$.
Let $F$ be the set of vertices $v$ that are incident to at least $\eps n$ modified pairs.
Then $|F| \le \eps n$. We divide the remainder of the proof into two cases,
according to whether $G' \in \mc{B}_a(n)$ or $G' \in \mc{R}_{b-1}(n)$.

\nib{Case 1: $G' \in \mc{B}_a(n)$.} 
Note that $p \le \tfrac{a}{a+b-1}$ by definition of $\mc{G}_p(n)$.
Let $(P'_1,\dots,P'_a)$ be the partition of $V(G')$ such that each $P'_i$ is a blue clique in $G'$.
Note that $||P'_i|-n/a| < \eps n/p$ for all $i$, otherwise, by Fact \ref{turansize}, we have 
$w_p(G') \le (1-p) \tbinom{n}{2} + p((1-1/a)n^2-(\eps n/p)^2)/2$, so by Fact \ref{extremalweights},
we have $w_p(G) \le \eps^2 n^2/2 + w_p(G') < w_p(B_a(n))$, which is a contradiction.
Let $P^*_i = P'_i \sm F$ for $1 \le i \le a$ and let $(P_1,\dots,P_a)$ be a partition of $V(G)$ 
such that $P^*_i \sub P_i$ for $1 \le i \le a$ minimising the number of red edges of $G$ inside parts 
(i.e.\ minimising $\sum_{i=1}^a |G[P_i]_r|$).
Since $|F| < \eps n$ we have $||P_i|-n/a| < (p^{-1}+1)\eps n$ for all $i$.

Now we prove a series of statements in the following claim
that gradually pin down the location of the red edges.
The final statement is that the parts $P_1,\dots,P_a$ each contain no red edges;
this implies $w_p(G) \le w_p(B_a(n))$, with equality only if $G=B_a(n)$,
so it will complete the proof in this case.
Note that all colours and degrees are defined with respect to $G$, not $G'$.

\nib{Claim.} Let $\{i,j\} \sub [a]$ with $i \neq j$.
\begin{enumerate}[label=(\roman*), nolistsep]
\item For any $x \in P^*_j$ we have $d_r(x,P^*_j) \le \eps n$.
\item For any vertex $x$ we have $d_r(x) \ge \tfrac{a-1}{a} n - \tfrac{4}{p}$. 
\item For any $x \in P^*_j$ we have $d_r(x,P^*_i) > |P^*_i| - (p^{-1}+4)\eps n$.
\item For any $x \in P_j$ we have $d_r(x,P^*_i) > |P^*_i| - 2n/3a$.
\item For any $x \in P_j$ we have $d_r(x,P^*_j) = 0$.
\item For any $x \in P_j$ we have $d_r(x,P^*_i) > |P^*_i| - (p^{-1}+3)\eps n$.
\item For any $x \in P_j$ we have $d_r(x,P_j) = 0$.
\end{enumerate}

\nib{Proof of Claim.}
For (i), note that if $y \in P^*_j$ and $xy$ is red,
then $xy$ must be modified, so by definition of $F$ we have $d_r(x,P^*_j) \le \eps n$.\\
For (ii), we note that $d_p(x) \ge (1-\tfrac{p}{a})n-4$ and apply Fact \ref{colordeg}.\\
For (iii), note that by (i) there are at least $|P^*_j| - \eps n \ge n/a - (p^{-1}+3)\eps n$ 
vertices $y$ in $P^*_j$ such that $xy$ is not red, so by (ii) the other parts have 
at most $(p^{-1}+3)\eps n + 4/p < (p^{-1}+4)\eps n$ such vertices.\\
For (iv), first note that we can assume $x \in F$, otherwise we are done by (iii).
Now suppose for a contradiction that we have $2n/3a$ vertices $y$ in $P^*_i$ such that 
$xy$ is not red. By (ii), this leaves at most $n/3a+4/p$ such vertices in $P_j$. 
Then $d_r(x,P_j) \ge |P_j| - (n/3a+4/p) > |P_i| - 2n/3a \ge d_r(x,P_i)$,
so moving $x$ to $P_i$ contradicts the minimality of $\sum_{i=1}^a |G[P_i]_r|$.\\
For (v), suppose for a contradiction that $xy$ is red for some $y \in P^*_j$.
By (iii) and (iv), we can choose sets $Z_k \sub P^*_k$ of size $n/6a$ for all $k \ne j$
such that $xz$ and $yz$ are red for all $z \in Z_k$. By Lemma \ref{partiteturanb}
applied to the graph of modified pairs, we can find a set $S$ with $b$ points in each $Z_k$
such that no pair inside $S$ is modified. However, this contradicts 
Fact \ref{findIa} (with $P^*_j$ playing the 
role of the $P_i$ excluded in the statement).\\
Statement (vi) follows from (v) and (ii) in the same way that (iii) followed from (i) and (ii).\\
Statement (vii) follows from (vi) in the same way that (v) followed from (iii) and (iv).
This proves the claim, and so completes the proof of Case 1.

\nib{Case 2: $G' \in \mc{R}_{b-1}(n)$.} 
The proof of this case is similar but slightly different to that of Case 1.
Note that $p \ge \tfrac{a}{a+b-1}$ by definition of $\mc{G}_p(n)$.
Let $(P'_1,\dots,P'_{b-1})$ be the partition of $V(G')$ such that each $P'_i$ is a red clique in $G'$.
Note that $||P'_i|-n/(b-1)| < \eps n/(1-p)$ for all $i$, similarly to Case 1.
Let $P^*_i = P'_i \sm F$ for $1 \le i \le b-1$ and let $(P_1,\dots,P_{b-1})$ be a partition of $V(G)$ 
such that $P^*_i \sub P_i$ for $1 \le i \le b-1$ minimising the number of blue edges of $G$ inside parts.
Since $|F| < \eps n$ we have $||P_i|-n/(b-1)| < ((1-p)^{-1}+1)\eps n$ for all $i$.
Now we gradually pin down the location of the blue edges;
the final statement of the following claim will complete the proof of Case 2, and so of the theorem.

\nib{Claim.} Let $\{i,j\} \sub [b-1]$ with $i \neq j$.
\begin{enumerate}[label=(\roman*), nolistsep]
\item For any $x \in P^*_j$ we have $d_g(x,P^*_j) \le d_b(x,P^*_j) \le \eps n$.
\item For any vertex $x$ we have $d_r(x) \ge \tfrac{a-1}{a} n - \tfrac{4}{p}$
and $d_b(x) \ge \tfrac{b-2}{b-1} n - \tfrac{4}{1-p}$. 
If $p \ge 1/2$ we also have $d_g(x) \ge \tfrac{b-2}{b-1} n - \tfrac{4}{1-p}$.
If $p \le 1/2$ and $x \notin F$ we have $d_g(x) \ge \tfrac{b-2}{b-1} n - ((1-p)^{-1}+3)\eps n/p$.
\item For any $x \in P^*_j$ we have $d_b(x,P^*_i) \ge d_g(x,P^*_i) > |P^*_i| - 2((1-p)^{-1}+3)\eps n/p$.
\item For any $x \in P_j$ we have $d_b(x,P^*_i) > |P^*_i| - 2n/3(b-1)$,
and if $p \ge 1/2$ we have $d_g(x,P^*_i) > |P^*_i| - 2n/3(b-1)$.
\item For any $x \in P_j$ we have $d_b(x,P^*_j) = 0$.
\item For any $x \in P_j$ we have $d_b(x,P^*_i) > |P^*_i| - ((1-p)^{-1}+3)\eps n$,
and if $p \ge 1/2$ we have $d_g(x,P^*_i) > |P^*_i| - ((1-p)^{-1}+3)\eps n$.
\item For any $x \in P_j$ we have $d_b(x,P_j) = 0$.
\end{enumerate}

\nib{Proof of Claim.}
Statement (i) follows from the definition of $F$. \\The first two parts of (ii) follow from 
Fact \ref{colordeg}, as $d_p(x) \ge (1-\tfrac{1-p}{b-1})n-4$. For the third, by (i) we have
$(1-\tfrac{1-p}{b-1})n-4 \le d_p(x) \le d_g(x) + p(|P^*_j|-\eps n) + (1-p)(n-d_g(x)-|P^*_j|+\eps n)$.
Since $|P^*_j|-\eps n \ge n/(b-1) - ((1-p)^{-1}+2)\eps n$ we obtain
\[pd_g(x) \ge (1-\tfrac{1-p}{b-1})n-4 - (1-p)n + (1-2p)(n/(b-1) - ((1-p)^{-1}+2)\eps n)
= p\tfrac{b-2}{b-1}n - 4 - (1-2p)((1-p)^{-1}+2)\eps n),\]
which implies the stated bound on $d_g(x)$. \\Now (iii) follows from (i) and (ii).\\
Statement (iv) follows from minimality of $\sum_{i=1}^a |G[P_i]_b|$ similarly to Case 1.

For (v), suppose for a contradiction that $xy$ is blue for some $y \in P^*_j$.
We consider two subcases according to whether $p \ge 1/2$. Suppose first that $p \ge 1/2$.
By (iii) and (iv), we can choose sets $Z_k \sub P^*_k$ of size $n/6b$ for all $k \ne j$
such that $xz$ and $yz$ are green for all $z \in Z_k$. By Lemma \ref{partiteturanb}
applied to the graph of modified pairs, we can find a set $S$ with $a+1$ points in each $Z_k$
such that no pair inside $S$ is modified. However, 
this contradicts Fact \ref{findIb}
(with $P_j$ playing the 
role of the $P_i$ excluded in the statement),
so (v) holds in the subcase $p \ge 1/2$. Now suppose that $p<1/2$.
As $p \ge \tfrac{a}{a+b-1}$ we have $a<b-1$, so as $a \ge 2$
we have $\tfrac{a-1}{a} \ge \tfrac{3}{2(b-1)}$.
By (i) and (ii) the number of common red neighbours of $x$ and $y$ outside of $P^*_j$ 
is at least \begin{align*} d_r(x) + d_g(y) - n - d_g(y,P^*_j) 
& \ge \tfrac{a-1}{a} n - \tfrac{4}{p} + \tfrac{b-2}{b-1} n - ((1-p)^{-1}+3)\eps n/p - n - \eps n \\
& \ge \tfrac{n}{2(b-1)} - ((1-p)^{-1}+5)\eps n/p \ge n/3b. \end{align*}
Then for some $j' \ne j$ we can fix $Z'_{j'} \sub P^*_{j'}$ of size $n/3b^2$
such that $xz$ and $yz$ are red for all $z \in Z'_{j'}$.
Also, by (iii) and (iv), we can choose sets $Z_k \sub P_k$ of size $n/3b^2$ for all $k \ne j$
such that $xz$ and $yz$ are blue for all $z \in Z_k$, and $Z_{j'} \cap Z'_{j'} = \es$.
By Lemma \ref{partiteturanb} applied to the graph of modified pairs, we can find a set $S$ 
with $a+1$ points in each $Z_k$ and in $Z'_{j'}$ such that no pair inside $S$ is modified. 
However, this contradicts Fact \ref{findIb} (with $P^*_j$ playing the 
role of the $P_i$ excluded in the statement).

Finally, (vi) follows from (v) and (ii), and (vii) follows by the same proof as for (v),
using (vi) in place of (iii) and (iv). Thus we have finished the proof of the theorem. \qed

\section{Concluding remarks}

There are several natural directions for future research arising from our paper.
As discussed in the introduction, the typical structure of graphs in a property 
(with various meanings of `typical') is much better understood for 
monotone than for hereditary properties, so one may ask to close this gap.
However, the results of Marchant and Thomason \cite{MT} show that 
the associated extremal problem for two weighted colours can be very difficult to analyse 
even for examples that at first sound innocuous (e.g.\ forbidding an induced $K_{3,3}$).
Nevertheless, we expect that there are other simple examples besides 'eyes'
that can be treated by methods similar to those in our paper.

There is much left to be done even just for eyes.
Firstly, there is the case $a=1$. As remarked earlier, Lemma \ref{uniquetype} fails in this case.
Indeed, we consider the following example from \cite[Theorem 3.27]{MT}.
Let $p=1/b$ and $\tau$ be any type whose vertices are red and in which the green edges 
form a connected $(b-2)$-regular graph, the other edges being blue. Then $\tau$ is $p$-core, 
does not describe $I_{1,b}$ and $\lL_p(\tau)=1-1/b=\kK_p(I_{1,b})$.
Furthermore, writing $C_\tau(n)$ for the corresponding blowup construction on $n$ vertices,
if $\tau$ has $k$ vertices, and $n$ is divisible by $k$ but not by $b-1$,
then calculations as in Fact \ref{extremalweights} show that
$w_p(C_\tau(n)) > w_p(R_{b-1}(n)) > w_p(B_1(n))$,
so Theorem \ref{exact} also fails for $a=1$.
This suggests that Theorem \ref{eye-free} may also fail in this case,
although our method of proof is not precise enough to decide this.


Secondly, there is the problem of obtaining the analogue of \cite{BMSW} for eyes:
for which $m$ is it true that with high probability a random $I_{a,b}$-ifree graph $G$ 
with $m$ edges satisfies $\chi(G)=a$ or $\chi(\ov{G})=b-1$?
Note that in this question we have altered Theorem \ref{eye-free} in two ways:
we have changed the random graph model from binomial to fixed size,
and have strengthened the conclusion by requiring the colouring properties
for $G$ itself, rather than some graph close to $G$. It is natural to think that 
this question could be answered by adapting the methods of \cite{BMSW}.
However, even the change of model seems to pose some difficulties:
while the natural coupling between the $p$-binomial and $p\tbinom{n}{2}$-edge models
gives them very similar behaviour for monotone properties,
this is not at all clear for hereditary properties,
and moreover, direct calculations seem to lead to a harder extremal problem.
Indeed, suppose that $G \sim G_{n,m}$ is a uniformly random graph
with $n$ vertices and $m = pN$, writing $N = \tbinom{n}{2}$.
Let $C$ be an igraph on $n$ vertices with $xN$ green edges and $yN$ (only) red edges.
Then $\mb{P}(G \sub C) = \tbinom{xN}{(p-y)N} \tbinom{N}{m}^{-1}$,
so $N^{-1} \log_2 \mb{P}(G \sub C) \sim xH((p-y)/x)-H(p)$, where $H$ is the entropy function.
Thus we are led to a nonlinear optimisation problem in two weighted colours,
as opposed to the linear problem that arises for the binomial model.

This naturally leads us to generalise the coloured extremal problem.
Rather than optimising some particular function of (red, blue),
can one describe the two-dimensional region consisting of their possible values?
(More precisely, we are interested in proportional values achievable for arbitrarily large $n$.)
Suppose $C$ is an $I_{a,b}$-free igraph on $[n]$ with $RN$ red edges and $BN$ blue edges.
Our results show that $(R,B)$ is within the region bounded by the lines
$R=1$, $B=1$, $R+B=1$ and $(b-1)R+aB=a+b-2$ (ignoring $o(1)$ errors).
Moreover, we can achieve $(1,1-1/a)$ and $((b-2)/(b-1),1)$
but no interior point of $(b-1)R+aB=a+b-2$ (by Theorem \ref{stability}), 
so the region is not convex. 
A description of the (red, blue) region would enable one to optimise any two variable function, 
such as the nonlinear function problem described above for the fixed size model,
or the minimum of $R$ and $B$, which was considered by Diwan and Mubayi \cite{DM},
for the more general problem in which white pairs are allowed.
Our proof of Theorem \ref{exact} still works allowing $o(n^2)$ white pairs,
but the result may well be (roughly) true allowing any number of white pairs;
this is a reformulation of \cite[Conjecture 16]{DM},
we say `roughly' because allowing white pairs introduces another possible extremal example,
namely a green complete $(a+b-1)$-partite graph with white pairs inside the parts.

\medskip

\nib{Acknowledgements.}
We thank the anonymous referees for helpful comments and corrections.

\end{document}